\documentclass[reqno]{amsart}

\usepackage{amsmath}
\usepackage{amssymb}
\usepackage{amsfonts}
\usepackage{graphicx}
\usepackage{amsthm}
\usepackage{enumerate}
\usepackage[mathscr]{eucal}
\usepackage{lscape}
\usepackage{dsfont}
\usepackage{color}
\usepackage{mathtools}

\usepackage[colorlinks]{hyperref}
\usepackage{setspace}
\newtheorem{theor}{Theorem}[section]

\newtheorem{remar}[theor]{Remark}

\newtheorem{corol}[theor]{Corollary}

\newtheorem{Exa}[theor]{Example}

\newcommand{\tyz}{Tian-Yau-Zelditch }

\newcommand{\Hm}{\mathcal{H}_m}

\newcommand{\R}{\mathbb{R}}

\newcommand{\C}{\mathbb{C}}

\newcommand{\N}{\mathbb{N}}

\newcommand{\K}{K\"{a}hler\ }

\newcommand{\Ric}{\textrm{Ric}}

\begin{document}
\title[The Simanca metric admits a regular quantization]{The Simanca metric admits a regular quantization}

\author{Francesco Cannas Aghedu, Andrea Loi}
\address{Dipartimento di Matematica e Informatica, Universit\`a di Cagliari\\Via Ospedale 72, 09124 Cagliari (Italy)}
\email{fcannasaghedu@unica.it, loi@unica.it}

\thanks{The second author was  supported by Prin 2015 -- Real and Complex Manifolds; Geometry, Topology and Harmonic Analysis -- Italy  by INdAM. GNSAGA - Gruppo Nazionale per le Strutture Algebriche, Geometriche e le loro Applicazioni and by Fondazione di Sardegna and Regione Autonoma della Sardegna.}

\begin{abstract}
Let $g_S$ be the Simanca metric on the blow-up $\tilde\C^2$ of $\C^2$ at the origin. We show that $(\tilde\C^2,g_S)$ admits a regular quantization. We use  this fact  to prove that all coefficients in the Tian-Yau-Zelditch expansion for the Simanca metric  vanish and that a dense subset of 
$(\tilde\C^2, g_S)$  admits a Berezin quantization.  
\end{abstract}
 
\maketitle
\tableofcontents
\newpage

\section{Introduction}
Let $M$ be a $n$-dimensional complex manifold endowed with a K\"ahler metric $g$. Assume that
there exists a holomorphic line bundle $L$ over $M$ such that $c_1(L) = [\omega]$, where $\omega$ is the K\"ahler form associated to $g$ and $c_1(L)$ denotes the first Chern class of $L$.  A necessary and sufficient condition for the existence of such an  $L$ is that $\omega$ is an integral \K form.

Let $m \geq 1$ be an integer and let $h_m$ be an Hermitian metric on $L^m = L^{\otimes m}$ such that its Ricci curvature $\Ric(h_m) = m\omega$. Here $\Ric(h_m)$ is the two-form on $M$ whose local expression is given by
\[
\Ric(h_m)= -\frac{i}{2\pi}\partial\bar{\partial}\log h_m(\sigma(x),\sigma(x))
\]
for a trivializing holomorphic section $\sigma : U \rightarrow L^m\setminus \{0\}$. In the quantum mechanics terminology
$L^m$ is called the \textit{prequantum line bundle}, the pair $(L^m, h_m)$ is called a \textit{geometric
quantization} of the K\"ahler manifold $(M,m\omega)$ and $\hbar= m^{-1}$ plays the role of Planck's constant
(see e.g. \cite{AreLoi}). Consider the separable complex Hilbert space $\Hm$ consisting of global
holomorphic sections $s$ of $L^m$ such that
\[
\left\langle s,s \right\rangle_{h_m}= \int_{M} h_m(s(x),s(x))\frac{\omega^n}{n!} < \infty.
\]
Define
\[
\epsilon_{mg}(x)= \sum_{j=0}^{d_m}h_m(s_j(x),s_j(x)),
\]
where $s_j,\,j=0,\ldots,d_m (\dim \Hm = d_m+1\leq \infty)$ is an orthonormal basis of $(\Hm,\left\langle \cdot, \cdot \right\rangle_{h_m})$.
As suggested by the notation this function depends only on the metric $mg$ and not on
the orthonormal basis chosen or on the Hermitian metric $h_m$. Obviously if $M$ is compact $\Hm=H^0(L^m)$, where $H^0(L^m)$ is the (finite dimensional) space of global holomorphic section of $L^m$.

In the literature the function $\epsilon_{mg}$ was first introduced under the name of $\eta$-function by Rawnsley in \cite{Rawnsley} later renamed as $\theta$\textit{-function} in \cite{CahGut}. 
A metric $g$ on $M$ is called {\em balanced}  if  $\epsilon_g$ is a positive constant.
The definition of balanced metrics was originally given by Donaldson \cite{Donaldson}   in the case of compact polarized \K manifold $(M,g)$ and generalized in \cite{AreLoi3} to the noncompact case (see also \cite{Englis2}, \cite{GreLoi},  \cite{LZ12}). \\
A geometric quantization $(L^m,h_m)$ of a K\"ahler manifold $(M,\omega)$ is called a  \textit{regular quantization}\footnote{Regular quantizations were introduced and studied  in \cite{CahGut} in the context of quantization by  deformation of \K manifolds.} 
if $mg$ is balanced   for any (sufficiently large) natural number $m$.

Many authors (see, e.g. \cite{AreLoi} and  \cite{AreLoiZ} and references therein) have  tried to understand what kind of properties are enjoyed by  those  K\"ahler manifolds which admit a regular quantization. 
Here we recall the following two facts:
\begin{itemize}
\item
a K\"ahler metric  which admits a regular quantization is cscK (constant  scalar curvature K\"{a}hler) metric (see \cite{Loi4});
\item
a geometric quantization of a  homogeneous\footnote{A K\"ahler manifold $(M, g)$ is homogeneous if the group $\mathrm{Aut}(M) \cap \mathrm{Isom}(M,g)$ acts transitively on $M$, where $\mathrm{Aut}(M)$ denotes the group of holomorphic diffeomorphisms of $M$ and $\mathrm{Isom}(M,g)$ the isometry group of $(M,g)$.} and simply-connected  K\"ahler manifold is regular (see \cite{AreLoi} and \cite{LM15}  for a proof in the compact and noncompact case respectively). 
\end{itemize}

Therefore,  the  following question naturally arises: \textit{Is it true that a complete\footnote{The assumption of completeness is necessary otherwise one can construct regular quantizations on non-homogeneous \K manifolds obtained by deleting a measure zero set from  a homogeneous \K manifold (see \cite{LMZ17}).} K\"ahler manifold $(M,\omega)$ which admits a regular quantization is necessarily homogeneous (and simply-connected)\footnote{The simply-connected request is in brackets since one can prove that every homogeneous and projectively induced \K manifold is simply-connected (see \cite{DiScala}).}? }

In the compact case this question is still open and of great interest also because the \K manifolds involved are projectively algebraic. 

In this paper we give a negative answer to the question in the non-compact case  by considering the Simanca metric $g_S$ on the blow-up $\tilde\C^2$ of $\C^2$ at the origin. The Simanca metric is  a well-known and important example (both from mathematical and physical point of view)  of non homogeneous complete, zero constant scalar curvature metric (see Section \ref{section2} below for details).  \\

Our main result is then the  following:
\begin{theor}\label{Th1}
Let  $\tilde\C^2$ be the blow-up of $\C^2$ at the origin endowed with the Simanca metric $g_S$. 
Then $(\tilde\C^2, g_S)$ admits a regular quantization such that  $\epsilon_{mg_S}=m^2$.
\end{theor}

It is worth pointing out that recently Bi-Feng-Tu  \cite{fbh} have constructed examples of regular quantizations on  Fock--Bargmann--Hartogs domains  in the complex Euclidean space equipped with a negative constant scalar curvature  \K metric. Thus, they  provide  a negative answer to the previous question in the non-compact case when the scalar curvature is negative. Another important difference between Bi-Feng-Tu example and the Simanca metric
is that in the first case the  quantization bundle is trivial and the \K metric has a global \K potential. Moreover,  the Simanca  metric has been a fundamental ingredient in the construction of cscK metrics on compact \K manifold via blow-up procedures (see \cite{AP09}). Thus we believe our Theorem \ref{Th1} could be used to built regular quantizations of non-homogeneous compact \K manifolds.

In order to obtain an interesting corollary of Theorem \ref{Th1} (see Corollary \ref{Cor2} below) 
we need to briefly recall  some important tools about  asympotic expansions of the epsilon function.
In the  case of a  compact  \K manifold $(M, \omega)$
Zelditch  \cite{Zelditch} proved that there exists a complete
asymptotic expansion in the $C^{\infty}$ category of epsilon  function:
\begin{equation}
\epsilon_{mg}(x) \sim \sum_{j=0}^{\infty}a_j(x)m^{n-j},
\label{TYZ}
\end{equation}
where $a_0(x)=1$ and $a_j(x),\, j=1,\ldots$ are smooth functions on $M$. 
The expansion \eqref{TYZ} is called \textit{Tian-Yau-Zelditch expansion} (TYZ in the sequel). Later on, Lu \cite{Lu}, by means of Tian's peak section method, proved that each of the coefficients $a_j(x)$ is a polynomial of the curvature and its covariant derivatives at $x$ of the metric $g$ which can be found by finitely many algebraic operations. In particular, he computed the first three coefficients.
The expression of the first two coefficients  is:
 \begin{equation}\label{coefflu}
\left\{\begin{array}
{l}
a_1(x)=\frac{1}{2}\rho\\
a_2(x)=\frac{1}{3}\Delta\rho
+\frac{1}{24}(|R|^2-4|{\rm Ric} |^2+3\rho ^2)\\
\end{array}\right.,
\end{equation}
where 
$\rho$, $R$, $\Ric$ denote respectively the scalar curvature,
the curvature tensor and the Ricci tensor of $(M, g)$.
The reader is also referred to \cite{Loi2} and \cite{Loi3} for a recursive formula for the coefficients $a_j$'s and an alternative computation of $a_j$ for $j \leq 3$ using Calabi's diastasis function (see also \cite{Xu} for a graph-theoretic interpretation of this recursive formula). 

It is  natural to study metrics with the \tyz coefficients being prescribed both in the compact that in the noncompact cases. For instance, the vanishing of this coefficients for large enough indexes  turns out to be related to some important problems in the theory of pseudoconvex manifolds (cf. \cite{lutian}). 
Furthermore, in  the noncompact case, one can find  in \cite{taubnut} a characterization of the flat metric as a Taub-NUT metric with $a_3=0$, while Z. Feng and Z. Tu \cite{fengtu} solve a conjecture formulated in \cite{zedda} by showing that  the complex hyperbolic space is the only Cartan-Hartogs domain where the coefficient  $a_2$ is constant. In \cite{LoiZe} A. Loi and M. Zedda prove that a locally hermitian symmetric space with vanishing $a_1$ and $a_2$ is flat. 

For the Simanca metric $(\tilde{\C}^2, g_S)$ in \cite{LoiZe} 
 the authors computed the $a_2$ coefficient  ($a_1=0$ since $g_S$ has vanishing scalar curvature).  Moreover, in \cite{LSZ} is proved that
 a projectively induced\footnote{A \K metric $g$ on a complex manifold $M$ is said to be \textit{projectively induced} if exists a holomorphic and isometric (i.e. K\"ahler) immersion of $(M,g)$ into the complex projective space $(\C P^N,g_{FS}), N \leq +\infty$, endowed with the Fubini-Study metric $g_{FS}$, the metric whose associated \K form is given in homogeneous coordinates by $\omega_{FS}= \frac{i}{2\pi}\partial\bar{\partial}\log(|Z_0|^2+ \cdots + |Z_N|^2)$.} radial  \K metric with  $a_1=a_3=0$  (or with  $a_1=a_2=0$) is  either the flat metric $g_0$ or the Simanca metric $g_S$.

\vskip 0.3cm 
Here, by using  Theorem \ref{Th1} we immediatly  obtain the following: 
\begin{corol}\label{Cor2}
All the coefficients $a_j(x)$, with $j \geq 1$, of  the TYZ expansion for the Simanca metric vanish. 
\end{corol}

Theorem \ref{Th1} shows that $(\C^2, g_0)$ and $(\tilde{\C}^2, g_S)$ 
have the same epsilon functions both equal to $m^2$.
It could be interesting to find other examples of \K manifolds sharing  this property and, more generally, to analyze to what extent the TYZ coefficients determine the underlying \K manifold  (cf. \cite{AreLoiZhombal} for this last issue).\\

We also prove a result on Berezin's quantization on the dense subset  $\C^2\setminus \{0\}\subset \tilde{\C}^2$ equipped with the restriction of the  Simanca \K form $\omega_S$ associated to the Simanca metric  $g_S$. This is  expressed by the following corollary.
\begin{corol}\label{Cor1}
$(\C^2\setminus \{0\}, \omega_S)$ admits a Berezin quantization.
\end{corol}

The construction in the proof of Theorem \ref{Th1} stops to work when $\C^2$ is replaced by $\C^n, n\geq 3$ and the metric $g_S$ is replaced by its natural generalization  $g_{S(n)}$ on $\tilde{\C}^n$ (see Section \ref{section6} for details). This is expressed by the following theorem.
\begin{theor}\label{Th3}
		Let  $\tilde\C^n$ be the blow-up of $\C^n$ at the origin endowed with the generalized Simanca metric $g_{S(n)}$.  For any integer $m\geq 1$ the following statements hold 
		\begin{enumerate}
			\item $(\tilde\C^n, mg_{S(n)})$ is projectively induced for any $n\geq 2$,
			\item $mg_{S(n)}$ is not balanced  for all $n \geq 3$.
		\end{enumerate}
\end{theor}	
That theorem gives an example of projectively induced \K metric $g$ on the blow up of $\C^n$ at the origin such that $mg$ is not balanced for any positive integer $m$.

The paper is organized as follows. Section \ref{section2} contains  basic facts on the Simanca metric and the proof of Theorem \ref{Th1}. In  Section \ref{section5}, after recalling the definition of  Berezin quantization, we prove  Corollary \ref{Cor1}. 
Finally in Section \ref{section6}, after describing  the  well-known link between balanced  and projectively induced metrics, we prove Theorem \ref{Th3}. \\

The authors would like to thank the referee for very helpful remarks  
and also  Michela Zedda for useful discussions.

\section{The Simanca metric and the proof of Theorem \ref{Th1}}\label{section2}
Let us briefly recall the definition of the  blow-up  $\tilde\C^2$ of $\C^2$ at the origin as 
\[
\tilde\C^2 = \{ (z_1, z_2, [t_1, t_2]) \in \C^2 \times \mathbb{C}P^{1} \ : \ t_1z_{2}-t_2z_{1} = 0\}.
\]

$\tilde\C^2$ is a closed submanifold of $\C^2 \times \mathbb{C}P^{1}$ of complex dimension $2$. 	
A system of charts for $\tilde\C^2$ is given as follows: for  $j =1,2$ we take 
\[
\tilde U_j = (\C^2 \times U_j) \cap \tilde{\C}^2,
\]
where $U_j=\{t_j \neq 0\}$, for $j =1, 2$,  are open subsets of $\mathbb{C}P^{1}$. Then we have two coordinate maps
\begin{equation*}
\begin{split}
\varphi_1 & : \tilde{U}_1 \rightarrow \mathbb{C}^2,\, \left(z_1,z_2,[t_1,t_2]\right) \mapsto \left(z_1,\frac{t_2}{t_1}\right), \\
\varphi_2 & : \tilde{U}_2 \rightarrow \mathbb{C}^2,\, \left(z_1,z_2,[t_1,t_2]\right) \mapsto \left(\frac{t_1}{t_2},z_2\right),
\end{split}
\end{equation*}
having as inverses the parametrization maps defined, respectively, by
\begin{equation}\label{chartsBUmap2}
\begin{split}
\varphi_1^{-1} & : \mathbb{C}^2 \rightarrow \tilde{U}_1,\, (w_1,w_2) \mapsto (w_1,w_1w_2,[1,w_2]), \\
\varphi_2^{-1} & : \mathbb{C}^2 \rightarrow \tilde{U}_2,\, (w_1,w_2) \mapsto (w_1w_2,w_2,[w_1,1]).
\end{split}
\end{equation}

There are two projection maps
\begin{align*}
p_1 & : \tilde \C^2 \rightarrow \C^2, \\
p_2 & : \tilde \C^2 \rightarrow \mathbb{C}P^{1},
\end{align*}
given by the restriction to $\tilde \C^2$ of the canonical projections of $\C^2 \times \mathbb{C}P^{1}$. One can prove (see \cite{McDuff}) that $p_2$ induces on  $\tilde{\C}^2$ the structure of complex line bundle, whose fibre over $[t_1,t_2] \in \mathbb{C}P^{1}$ is the corresponding line $\{(\lambda t_1, \lambda t_2)\, |\, \lambda \in \C\}$ in $\C^2$. In other words, this is the universal line bundle over $\mathbb{C}P^{1}$. Observe that $p_1$ is bijective when restricted to $p_1^{-1}(\C^2 \setminus \{0\})$, while
\[
p_1^{-1}(0) = \{(0, [t]) \in \tilde \C^2  \} \simeq \mathbb{C}P^{1}.
\]
Thus we may think of $\tilde\C^2$ as obtained from $\C^2$ by replacing the origin $0$ by the space of all lines in $\C^2$ through $0$. The manifold $p_1^{-1}(0)$ is called the \textit{exceptional divisor}, and we will denote it by $H$. So, the restriction 
\[
p_r := {p_1}_{|\tilde \C^2  \setminus H}: \tilde \C^2 \setminus H \rightarrow \C^2 \setminus \{ 0 \}, \ \ (z, [t]) \mapsto z
\]
is a biholomorphism, having as inverse
\[
\C^2 \setminus \{ 0 \} \rightarrow \tilde \C^2  \setminus H, \ \ z \mapsto \left( z_1,z_2, \left[z_1,z_2\right] \right).
\]
Take now on $\C^2 \setminus \{0\}$ the $(1,1)$-form given by 
\begin{equation}
\omega = \frac{i}{2\pi} \partial \bar \partial(|z|^2+ \log|z|^2).
\label{AWS}
\end{equation}
We claim that the pull-back $p_r^*(\omega)$ of $\omega$, a priori defined only on $\tilde\C^2 \setminus H$, extends in fact to all $\tilde\C^2$.

The pull-back $p_r^*(\omega)$ is given in the coordinates (\ref{chartsBUmap2}) by
\[
p_r^*(\omega) = \frac{i}{2\pi}\partial\bar{\partial}\left( |z_1|^2(1+|z_2|^2)+\log(1+|z_2|^2)\right),
\]
on $\tilde{U}_1 \setminus H$, and
\[
p_r^*(\omega) =  \frac{i}{2\pi}\partial\bar{\partial}\left( |z_2|^2(1+|z_1|^2)+\log(1+|z_1|^2)\right),
\]
on  $\tilde{U}_2 \setminus H$. This shows that $p_r^*(\omega)$ extends to the whole $\tilde\C^2$, as claimed.  Clearly on $\tilde\C^2 \setminus H$ this form is given in local coordinates by \eqref{AWS}. The metric associated to \eqref{AWS} was introduced for the first time by S. Simanca in \cite{Simanca} and it is known in literature as the Simanca metric and denoted here by $g_S$. The form \eqref{AWS} is denoted here by $\omega_S$. It is not hard to see that $g_S$ is a complete, zero constant scalar curvature \K metric (\cite{Simanca}). Moreover $(\tilde\C^2,\omega_S)$ is a non homogeneous manifold (see, e.g. \cite{LoiZe}). 	
Notice also  that $\tilde \C^2$ is  \emph{non contractible} but simply-connected since $\tilde\C^2$ is diffeomorphic to the connected sum
$\C^2 \# \overline{\mathbb{C}P^2}$,
where $\overline{\mathbb{C}P^2}$ is the complex projective space ${\mathbb{C}P^2}$ with the opposite orientation.

\vskip 0.3cm 
In order to prove Theorem \ref{Th1} consider the holomorphic line bundle $L \rightarrow \tilde\C^2$ such that $c_1(L)=[\omega_S]$, where $c_1(L)$ is the first Chern class of $L$.
Such line bundle exists since $\omega_S$ is integral and it is unique, up to isomorphisms of line bundle, since $\tilde\C^2$ is simply-connected.
Notice that this line  bundle $L$ is not trivial and the \K form $\omega_S$ associated to $g_S$ does not admit a global \K potential in contrast with the  example of  Bi-Feng-Tu discussed in the introduction.

One can easily verify that the holomorphic line bundle $L^m  \rightarrow \tilde\C^2$ equipped with the hermitian structure  
\[
h_m(\sigma(x),\sigma(x))= \frac{1}{|z|^{2m}}e^{-m|z|^2}|q|^2,
\]
defines a geometric quantization of $(\tilde\C^2,m\omega_S)$, where $m$ is a positive natural number and $\sigma: U \subset \tilde{\C}^2 \setminus H  \rightarrow L^m\setminus \{0\}, x \mapsto (z,q) \in U \times \C$ is a trivialising holomorphic section. Since $L^m_{|\C^2\setminus \{0\}}$ is equivalent to the trivial bundle $\C^2\setminus \{0\} \times \C$, one can find  a natural bijection between the complex space $H^0(L^m)$ and the space of holomorphic functions on $\C^2$ vanishing at the origin with order greater or equal than $m$ 
(see, e.g.  \cite[Chapter 1]{GH}). This bijection takes $s\in H^0(L^m)$ to the holomorphic function $f_s$ on $\C^2$ obtained by restricting $s$ to $\tilde \C^2  \setminus H \simeq \C^2 \setminus \{0\}$. Moreover, since $H$ has zero measure in $\tilde\C^2$, one gets 
\begin{equation}\label{intnorm2}
\begin{split}
\left\langle s,s \right\rangle_{h_m} & = \int_{\tilde \C^2 } h_m(s(x),s(x))\frac{\omega_S^2}{2!} =\\
& = \int_{\C^2 \setminus \{0\}} \frac{e^{-m|z|^2}}{|z|^{2m}}|f_s(z)|^2 \left(1+\frac{1}{|z|^2} \right) d\mu (z)<\infty,
\end{split}
\end{equation}
where $d\mu (z)= \left(\frac{i}{2\pi} \right)^2 dz_1 \wedge d\bar{z}_1\wedge dz_2 \wedge d\bar{z}_2$. Therefore, in this case, the inclusion $\mathcal{H}_m\subseteq H^0(L^m)$ is indeed an equality, namely $\mathcal{H}_m=H^0(L^m)$
\footnote{Here $\mathcal{H}_m$ (as in  the introduction) denotes the space of global holomorphic sections $s$ of $L^m$, which are bounded with respect to
\[
\left\langle s,s \right\rangle_{h_m}=||s||^2_{h_m} = \int_{\tilde{\C}^2} h_m(s(x),s(x))\frac{\omega_S^2}{2!}.
\]}.
\newline

We are now ready to prove Theorem \ref{Th1}.
\begin{proof}[Proof of Theorem \ref{Th1}]
By passing to polar coordinates  $z_1=\rho_1 e^{i\vartheta_1}, z_2=\rho_2 e^{i\vartheta_2}$ with $\rho_1,\rho_2\in (0,+\infty), \vartheta_1,\vartheta_2 \in (0,2\pi)$  one  easily sees that the monomials $\{z_1^{j}z_2^{k}\}_{j+k \geq m}$ form a complete  orthogonal system for the Hilbert space $(\Hm,\left\langle \cdot, \cdot \right\rangle_{h_m})$.
Moreover, by (\ref{intnorm2}), 
\[
||z_1^jz_2^k||^2_{h_m}= 4\int_0^{+\infty}\int_0^{+\infty}\frac{e^{-m(\rho_1^2+\rho_2^2)}}{(\rho_1^{2}+\rho_2^2)^{m+1}}(1+\rho_1^2+\rho_2^2)  \rho_1^{2j}\rho_2^{2k} \rho_1\rho_2d\rho_1d\rho_2.
\]
With the substitution $\rho_1=r\cos\theta$, $\rho_2=r\sin\theta$, $0<r<+\infty$, $0<\theta <\frac{\pi}{2}$ one finds a product of one variable integral:
\[
||z_1^jz_2^k||^2_{h_m} = 4 \int_0^{\frac{\pi}{2}}(\cos\theta )^{2j+1}(\sin\theta)^{2k+1}d\theta\cdot
\int_0^{+\infty} r^{2(j+k-m)+1} (1+r^2)e^{-mr^2}\,dr.
\]
For the first integral we have (see \cite[page 255 (6.1.1)]{AbrSteg})
\begin{equation}
\int_0^{\frac{\pi}{2}}(\cos\theta )^{2j+1}(\sin\theta)^{2k+1}d\theta=\frac{\Gamma (j+1)\Gamma (k+1)}{2\Gamma (j+k+2)}=
\frac{j!k!}{2(j+k+1)!}.
\label{F1}
\end{equation}
For the second intergral, by (see \cite[page 258 (6.2.1)]{AbrSteg})
\begin{equation}
\int_0^{\infty}r^se^{-mr^2}\,dr=\frac{\Gamma (\frac{s+1}{2})}{2m^{(\frac{s+1}{2})}},
\label{F2}
\end{equation}
we find
\begin{align*}
\int_0^{\infty} r^{2(j+k-m)+1} (1+r^2)e^{-mr^2}\,dr & =\frac{\Gamma (j+k-m+1)}{2m^{j+k-m+1}}+\frac{\Gamma (j+k-m+2)}{2m^{j+k-m+2}}= \\
& =\frac{(j+k-m)!(j+k+1)}{2m^{j+k-m+2}}.
\end{align*}
Hence one gets
\begin{equation} 
||z_1^jz_2^k||^2_{h_m}=4\left[\frac{j!k!}{2(j+k+1)!}\right]\cdot \left[\frac{(j+k-m)!(j+k+1)}{2m^{j+k-m+2}}\right]=\frac{j!k!}{(j+k)!}\frac{(j+k-m)!}{m^{j+k-m+2}}.
\end{equation}
Therefore
\[
\left\lbrace \frac{z_1^jz_2^k}{||z_1^jz_2^k||_{h_m}} \right\rbrace_{j+k \geq m} 
\]
is an orthonormal basis for the Hilber Space $(\Hm,\left\langle \cdot, \cdot \right\rangle_{h_m})$. 

For the epsilon function one obtains
\begin{align*}
\epsilon_{mg} (z) & =\sum_{\substack{j, k\geq 0{}\\j+k\geq m}}\frac{e^{-m(|z_1|^2+|z_2|^2)}}{(|z_1|^2+|z_2|^2)^m}\frac{|z_1|^{2j}|z_2|^{2k}}{||z_1^jz_2^k||^2_{h_m}}= \\
& = \frac{e^{-m(|z_1|^2+|z_2|^2)}}{(|z_1|^2+|z_2|^2)^m}\sum_{\substack{j, k\geq 0{}\\j+k\geq m}}\frac{(j+k)!|z_1|^{2j}|z_2|^{2k}}{j!k! (j+k-m)!}m^{j+k-m+2}= \\
& =\frac{e^{-m(|z_1|^2+|z_2|^2)}}{(|z_1|^2+|z_2|^2)^m}\sum_{\beta =m}^{\infty}\left[\sum_{\substack{j, k\geq 0{}\\ j+k= \beta}}\frac{(j+k)!|z_1|^{2j}|z_2|^{2k}}{j!k!}\right]\frac{m^{\beta-m+2}}{(\beta-m)!} = \\
& =\frac{e^{-m(|z_1|^2+|z_2|^2)}}{(|z_1|^2+|z_2|^2)^m}\sum_{\beta =m}^{\infty} (|z_1|^2+|z_2|^2)^{\beta}\frac{m^{\beta-m+2}}{(\beta-m)!}= \\
&=m^2e^{-m(|z_1|^2+|z_2|^2)}\sum_{\beta =m}^{\infty} (|z_1|^2+|z_2|^2)^{\beta -m}\frac{m^{\beta-m}}{(\beta-m)!}= \\
& = m^2e^{-m(|z_1|^2+|z_2|^2)}\sum_{\alpha =0}^{\infty} (|z_1|^2+|z_2|^2)^{\alpha}\frac{m^{\alpha}}{\alpha!}= \\
& = m^2,
\end{align*} 
and this proves the theorem.
\end{proof}

\section{Berezin quantization and the proof of Corollary \ref{Cor1}}\label{section5}
Let $(M,\omega)$ be a symplectic manifold and let $\{\cdot,\cdot\}$ be the associated Poisson bracket. A \textit{Berezin quantization} (we refer to \cite{Berezin} for details) on $M$ is given by a family of associative algebras $\mathcal{A}_{\hbar}$ where the parameter $\hbar$ (which plays the role of the Planck constant) ranges over a set $E$
of positive reals with limit point $0$. Then in the direct sum $\oplus_{\hbar \in E}\mathcal{A}_{\hbar}$ with component-wise
product $*$, there exists a subalgebra $\mathcal{A}$, such that for an arbitrary element $f = f (\hbar) \in A$,
where $f (\hbar) \in \mathcal{A}_{\hbar}$, there exists a limit $\lim_{\hbar \to 0} f (\hbar) = \varphi( f ) \in C^{\infty}(M)$. The following correspondence principle must hold: for $f,g \in \mathcal{A}$
\[
\varphi(f*g)=\varphi(f)\varphi(g),\quad \varphi(\hbar^{-1}(f*g-g*f))=i\{\varphi(f),\varphi(g)\}.
\]
Moreover, for any pair of points $x_1,x_2 \in M$ there exists $f \in \mathcal{A}$ such that $\varphi(f)(x_1) \neq \varphi(f)(x_2)$.

Consider now a real analytic K\"ahler manifold $M$, with K\"ahler metric $g$ and associated
K\"ahler form $\omega$. Assume that there exists a (real analytic) global K\"ahler potential $\Phi : M \rightarrow \R$.
This function extends to a sesquianalytic function $\Phi(x, \bar{y})$ on that neighbourhood of the diagonal
in $M \times M$ such that $\Phi(x, \bar{x}) = \Phi(x)$. Consider \textit{Calabi's diastasis function} $D_g$ defined
on a neighbourhood of the diagonal in $M \times M$ by:
\[
D_g(x,y)=\Phi(x,\bar{x}) + \Phi(y,\bar{y})-\Phi(x,\bar{y})-\Phi(y,\bar{x}).
\]
By its definition we see that Calabi's diastasis function is independent from the potential
chosen which is defined up to the sum with the real part of a holomorphic function. Moreover,
it is easily seen that $D_g$ is real-valued, symmetric in $x$ and $y$ and $D_g(x, x) = 0$ (the
reader is referred to \cite{Calabi} and \cite{Loi} for more details on the diastasis function).
\begin{Exa}\label{Exa1}\rm
Let $g_{FS}$ be the Fubini-Study metric on the infinite dimensional complex projective
space $\mathbb{C}P^{\infty}$ of holomorphic sectional curvature $4$ and let $D_{g_{FS}}(p, q)$ be the associated
Calabi's diastasis function. One can show that for all $p \in \mathbb{C}P^{\infty}$ the function $D_{g_{FS}}(p, \cdot)$ is
globally defined except in the cut locus $H_p$ of $p$ where it blows up. Moreover $e^{-D_{g_{FS}}(p, q)}$ is
globally defined and smooth on $\mathbb{C}P^{\infty}$, $e^{-D_{g_{FS}}(p, q)}\leq 1$ and $e^{-D_{g_{FS}}(p, q)}=1$ if and only if $p=q$ (see \cite{Loi} for details).	
\end{Exa}
The following theorem is a reformulation of Berezin quantization result (see  \cite{Englis} and \cite{LM12}) in terms of Rawnsley $\epsilon$-function and Calabi's diastasis function.
\begin{theor}\label{Th2}
	Let $\Omega \subset \C^n$ be a complex domain equipped with a real analytic K\"ahler form $\omega$ and corresponding K\"ahler metric $g$. Then, $(\Omega,\omega)$ admits a Berezin quantization if the following two conditions are satisfied:
	\begin{enumerate}
		\item Rawnsley's function $\epsilon_{m g}(x)$ is a positive constant for all sufficiently large $m$;
		\item the function $e^{-D_{g}(x, y)}$ is globally defined on $\Omega \times \Omega$, $e^{-D_g(x, y)}\leq 1$ and $e^{-D_g(x, qy)}=1$ if and only if $x=y$.
	\end{enumerate}
\end{theor}
We are now in the position to prove Corollary \ref{Cor1}, namely that $(\C^2\setminus \{0\},\omega_S)$ admits a Berezin quantization.
\begin{proof}[Proof of Corollary \ref{Cor1}]
We are going to show  that Conditions $1$ and $2$ of Theorem \ref{Th2} are fulfilled  by $(\C^2\setminus \{0\},\omega_S)$. 
Condition 1 follows by  Theorem \ref{Th1}. For Condition 2 consider the holomorphic map
\[
\varphi : \C^2 \setminus \{0\} \rightarrow 	\mathbb{C}P^{\infty}
\]
given by
\[
(z_1,z_2) \mapsto \left(z_1,z_2,\ldots, \sqrt{\frac{j+k}{j!k!}}z_1^jz_2^k,\ldots\right),\, j+k \neq 0. 
\]
It is not hard to see that $\varphi$ is an injective  K\"ahler immersion from $(\C^2\setminus \{0\}, g_S)$ into $(\mathbb{C}P^{\infty}, g_{FS})$ (see  \cite{LSZ} for a proof). 
By Example \ref{Exa1},
Calabi's diastasis function $D_{g_{FS}}$ of $\C P^{\infty}$ is such that  $e^{-D_{g_{FS}}}$
is globally defined on $\C P^{\infty}\times\C P^{\infty}$ and by the hereditary property of the diastasis function
 (see \cite[Proposition  6]{Calabi}) we get that, for all $x, y\in \C^2 \setminus \{0\}$,
\begin{equation}\label{fondequ}
e^{-D_{FS}(\varphi (x), \varphi (y))}=e^{-D_{g_S}(x, y)}
\end{equation}
 is globally defined on $\C^2 \setminus \{0\}\times \C^2 \setminus \{0\}$.
 Since, by Example \ref{Exa1},   $e^{-D_{FS}(p, q)}\leq 1$ for all $p, q\in \C P^{\infty}$ it follows that $e^{-D_{g_S}(x, y)}\leq 1$ for  all
 $x, y\in  \C^2 \setminus \{0\}$ and since $\varphi$ is injective one gets that   $e^{-D_{g_S}(x, y)}=1$ iff $x=y$. Hence, also Condition 2 is satisfied
 and this concludes the proof of the corollary.
\end{proof}

\section{Balanced metrics on the blow-up of $\C^n$ at the origin and the proof of Theorem \ref{Th3}}\label{section6}
It is well known (see \cite{CahGut} and \cite{Rawnsley}) that if $mg$ is a balanced metric, namely  the function $\epsilon_{mg}$ is a positive constant, then $mg$ is projectively induced via the coherent states map 
\[
\varphi_m : M \rightarrow \C P^{d_m},\, x \mapsto [s_0(x),\ldots, s_j(x),\ldots].
\]
In fact the relation between this map and the function $\epsilon_{mg}$ can be read in the following formula due to Rawnsley (see \cite{Rawnsley}):
\begin{equation}
\varphi_m^*(\omega_{FS}) = \frac{i}{2\pi} \partial \bar{\partial} \log  \sum _{j=0}^{d_m} |s_j(x)|^2 = m \omega+ \frac{i}{2\pi} \partial \bar{\partial} \log \epsilon_{mg}.  
\label{E10}
\end{equation}
Therefore, by Theorem \ref{Th1} the metric $mg_S$ on $\tilde\C^2$ is projectively induced for all  $m>0$.
Another example of \K metric  on the blow-up $\tilde\C^2$  of $\C^2$ at the origin is the celebrated Eguchi--Hanson metric $g_{EH}$, namely the complete Ricci flat metric on 
$\tilde\C^2$ whose \K form   on  $\tilde\C^2\setminus H\cong \C^2\setminus \{0\}$ is given by  
\begin{equation}
\omega_{EH} = \frac{i}{2\pi} \partial \bar \partial(\sqrt{|z|^4+1}+\log|z|^2-\log(1+\sqrt{|z|^4+1})).
\label{WS}
\end{equation}
Recently in \cite{FCA},  the first author of the present paper shows that the metric $mg_{EH}$ is not balanced for any positive integer $m$. Hence, it is natural to see if there exist  examples of complete metrics $g$ on the blow-up of the Euclidean space at the origin such that $mg$ is projectively induced  but $mg$ is not balanced for any $m>0$. In this section we construct such a metric.

Let  $g_{S(n)}$ be the \textit{generalized Simanca metric}  on the blow-up $\tilde\C^n$ of $\C^n$ at the origin  whose \K form on $\C^n \setminus \{0\}\cong \tilde{\C}^n \setminus H$ is given by
\[
\omega_{S(n)}=\frac{i}{2\pi} \partial \bar{\partial}(|z|^2+\log |z|^2),\ \   |z|^2=|z_1|^2+\cdots+|z_n|^2
\]
and  $H$ denotes  the exceptional divisor arising by the blow-up construction (as in Section \ref{section2} one can  show that  $\omega_{S(n)}$, a priori defined only on $\C^n \setminus \{0\}$, extends to all $\tilde{\C}^n$). When $n=2$, $g_{S(2)}=g_S$. Notice that the generalized Simanca metric is complete but its scalar curvature is not constant.
\begin{remar}\label{R1} In the case of the generalized Simanca metric, for $n=3$ we have
		\[
		g_{S(3)}=\begin{pmatrix}
		1+\frac{|z_2|^2+|z_3|^2}{(|z_1|^2+|z_2|^2+|z_3|^2)} & -\frac{\bar{z}_1z_2}{(|z_1|^2+|z_2|^2+|z_3|^2)} &-\frac{\bar{z}_1z_3}{(|z_1|^2+|z_2|^2+|z_3|^2)}\\
		-\frac{z_1\bar{z}_2}{(|z_1|^2+|z_2|^2+|z_3|^2)}      &  1+\frac{|z_1|^2+|z_3|^2}{(|z_1|^2+|z_2|^2+|z_3|^2)}  & -\frac{\bar{z}_2z_3}{(|z_1|^2+|z_2|^2+|z_3|^2)} \\
		-\frac{z_1\bar{z}_3}{(|z_1|^2+|z_2|^2+|z_3|^2)}& -\frac{z_2\bar{z}_3}{(|z_1|^2+|z_2|^2+|z_3|^2)} & 1+\frac{|z_1|^2+|z_2|^2}{(|z_1|^2+|z_2|^2+|z_3|^2)}
		\end{pmatrix}
		\]
		so that, for $(z_1,0,0)$
		\[
		g_{S(3)}=\begin{pmatrix}
		1  & 0 &0 \\
		0  &  1+\frac{1}{|z_1|^2}   & 0  \\
		0  & 0  & 1+\frac{1}{|z_1|^2}
		\end{pmatrix}, \quad g_{S(3)}^{-1}=\begin{pmatrix}
		1  & 0 &0 \\
		0  & \frac{|z_1|^2}{1+|z_1|^2}   & 0  \\
		0  & 0  & \frac{|z_1|^2}{1+|z_1|^2}
		\end{pmatrix}.
		\]
		By recalling that $\Ric_{i\bar{j}}=-\frac{\partial^2 \log \det g}{\partial z_i\partial \bar{z}_j}$ one gets:
		\[
		\Ric=\begin{pmatrix}
		-\frac{1}{(1+|z_1|^2)^2} & 0 &0 \\
		0  &  \frac{1}{|z_1|^2+|z_1|^4}   & 0  \\
		0  & 0  &  \frac{1}{|z_1|^2+|z_1|^4} 
		\end{pmatrix}.
		\]
		Finally, by recalling that $\rho_g=-\sum g^{i\bar{j}}\Ric_{i\bar{j}}$, one finds:
		\[
		\rho_{g_{S(3)}}= - \frac{1}{(1+|z_1|^2)^2}.
		\]
		More generally it is possible to prove that
		\[
		\rho_{g_{S(n)}}= \frac{2-n}{(1+|z_1|^2)^2},\quad \mbox{for}\quad (z_1,0,\ldots,0)
		\]
		so the scalar curvature is not costant for any $n\geq 3$. 
\end{remar}

\begin{proof}[Proof of Theorem \ref{Th3}] 1. 
The holomorphic map
	\[
	\varphi : \C^n \setminus \{0\} \rightarrow \C P^{\infty}
	\]
	given by
	\[
	(z_1,\ldots,z_n) \mapsto \left(z_1,\ldots,z_n,\ldots, \sqrt{\frac{j_1+\cdots + j_n}{j_1!\cdots j_n!}}z_1^{j_1}\cdots z_n^{j_n},\ldots\right),\, j_1+\cdots + j_n \neq 0,
	\]
	is a \K immersion from $(\C^n \setminus \{0\},g_{S(n)})$ into $(\C P^{\infty},g_{FS})$. In point of fact 
	\begin{align*}
	\varphi^*(\omega_{FS}) & = \frac{i}{2\pi}  \partial \bar{\partial} \log \left(	\sum_{\substack{j_1, j_2,\ldots, j_n \geq 0{}\\j_1+\cdots + j_n\geq m}}\left( \frac{j_1+\cdots + j_n}{j_1!\cdots j_n!}|z_1|^{2j_1}\cdots |z_n|^{2j_n} \right)\right) = \\ 
	& = \frac{i}{2\pi} \partial \bar{\partial} \log (e^{|z|^2}|z|^2) = \omega_{S(n)}.
	\end{align*}
	Since $\tilde{\C}^n$ is simply-connected it follows from a result of Calabi (see \cite{Calabi} and \cite{LSZ}) that $\varphi$ extends to a \K immersion from $(\tilde{\C}^n,g_{S(n)})$ into $(\C P^{\infty},g_{FS})$. Similarly, one can show that $mg_{S(n)}$ is projectively induced for any positive integer $m$. Indeed, by a  result of Calabi  \cite{Calabi} if a \K manifold can be \K immersed into $\C P^{\infty}$ then the same is true for $(M, mg)$.

	2. For  an integer $m>0$, consider the  geometric quantization given by the holomorphic line bundle  $L^m \rightarrow (\tilde{\C}^n,\omega_{S(n)})$ such that $c_1(L^m)=m[\omega_{S(n)}]$, equipped with the hermitian structure  
	\[
	h_m(\sigma(x),\sigma(x))= \frac{1}{|z|^{2m}}e^{-m|z|^2}|q|^2.
	\]
	where  $\sigma: U \subset \tilde{\C}^n\setminus H \rightarrow L^m\setminus \{0\}, x \mapsto (z,q) \in U \times \C$ is a trivialising holomorphic section. As for the Simanca metric  there is a natural bijection between the complex  space $H^0(L^m)$ of global holomorphic sections  and the space of  holomorphic functions on $\C^n$ vanishing at the origin with order greater or equal than $m$.
	This bijection takes $s\in H^0(L^m)$ to the holomorphic function $f_s$ on $\C^n$ obtained by restricting $s$ to $\tilde{\C}^n \setminus H \simeq \C^n \setminus \{0\}$. Moreover, since $H$ has zero measure in $\tilde{\C}^n$, one gets 
	
	\begin{equation}\label{intnorm}
	\begin{split}
	\left\langle s,s \right\rangle_{h_m} & = \int_{\tilde \C^n } h_m(s(x),s(x))\frac{\omega_{S(n)}^n}{n!} =\\
	& = \int_{\C^n \setminus \{0\}} \frac{e^{-m|z|^2}}{|z|^{2m}}|f_s(z)|^2 \left(1+\frac{1}{|z|^2} \right)^{n-1} d\mu (z)<\infty,
	\end{split}
	\end{equation}
	where $d\mu (z)= \left(\frac{i}{2\pi} \right)^n dz_1 \wedge d\bar{z}_1\wedge dz_2 \wedge d\bar{z}_2 \wedge \ldots \wedge dz_n \wedge d\bar{z}_n$. Therefore $\mathcal{H}_m=H^0(L^m)$. 
	
From \eqref{intnorm} by passing to polar coordinates  $z_1=\rho_1 e^{i\vartheta_1},\ldots, z_n=\rho_n e^{i\vartheta_n}$ with $\rho_1,\ldots \rho_n \in (0,+\infty), \vartheta_1,\ldots,\vartheta_n \in (0,2\pi)$  one  easily sees that the set
    \[
    \{z_1^{j_1}\cdots z_n^{j_n}\}, \quad j_1+\cdots + j_n \geq m
    \]
    form a complete  orthogonal system for the Hilbert space $(\Hm,\left\langle \cdot, \cdot \right\rangle_{h_m})$.  Moreover, by (\ref{intnorm}), 
   \[
   ||z_1^{j_1}\cdots z_n^{j_n}||^2_{h_m}= 2^n \int_{\Omega}\frac{e^{-m(\rho_1^2+\cdots + \rho_n^2)}}{(\rho_1^2+\cdots + \rho_n^2)^{m+n-1}}(1+\rho_1^2+\cdots + \rho_n^2)^{n-1}  \rho_1^{2j_1+1}\cdots \rho_n^{2j_n+1}\,d\rho_1\cdots d\rho_n,
   \]
   where $\Omega = \{(0,+\infty)^n \subset \R^n\}$. With the substitution 
   \begin{align*}
   & \rho_1    =  r \cos(\vartheta_1) \\
   & \rho_2   =   r \sin(\vartheta_1)\cos(\vartheta_2) \\
   & \rho_3  =  r \sin(\vartheta_1)\sin(\vartheta_2)\cos(\vartheta_3) \\
   &  \vdots   \\
   & \rho_{n-1}   =  r \sin(\vartheta_1)\cdots \sin(\vartheta_{n-2})\cos(\vartheta_{n-1}) \\
   & \rho_{n}   =   r \sin(\vartheta_1)\cdots \sin(\vartheta_{n-2})\sin(\vartheta_{n-1}) 
   \end{align*}
   with $0< r < +\infty,\, 0 < \vartheta_i < \frac{\pi}{2}$ for $i=1,\ldots,n-1$, one finds a product of one variable integrals:
   \begin{align*}
   ||z_1^{j_1}\cdots z_n^{j_n}||^2_{h_m}  =  & 2^n \int_0^{\frac{\pi}{2}}(\cos\theta_1 )^{2j_1+1}(\sin\theta_1)^{2(j_2+\cdots +j_n+(n-1)-1)+1}d\theta_1 \cdot \\
   & \cdot \int_0^{\frac{\pi}{2}}(\cos\theta_2 )^{2j_2+1}(\sin\theta_2)^{2(j_3+\cdots +j_n+(n-1)-2)+1}d\theta_2 \cdot \\
   & \cdot \int_0^{\frac{\pi}{2}}(\cos\theta_3 )^{2j_3+1}(\sin\theta_3)^{2(j_4+\cdots +j_n+(n-1)-3)+1}d\theta_3 \cdot \\
   & \vdots \\
   & \cdot \int_0^{\frac{\pi}{2}}(\cos\theta_{n-2} )^{2j_{n-2}+1}(\sin\theta_{n-2})^{2(j_{n-1} +j_n+(n-1)-(n-2))+1}d\theta_{n-2} \cdot \\
   & \cdot \int_0^{\frac{\pi}{2}}(\cos\theta_{n-1} )^{2j_{n-1}+1}(\sin\theta_{n-1})^{2(j_n+(n-1)-(n-1))+1}d\theta_{n-1} \cdot \\
   & \cdot \int_0^{+\infty} r^{2(j_1+\cdots + j_n-m)+1} (1+r^2)^{n-1}e^{-mr^2}\,dr.
   \end{align*}
   For the first $n-1$ integrals, by \eqref{F1} we find
   \begin{align*}
   & \int_0^{\frac{\pi}{2}}(\cos\theta_1 )^{2j_1+1}(\sin\theta_1)^{2(j_2+\cdots +j_n+(n-1)-1)+1}d\theta_1  = 
   \frac{j_1!(j_2+\cdots +j_n+n-2)!}{2(j_1+j_2+\cdots +j_n+n-1)!},\\
   & \int_0^{\frac{\pi}{2}}(\cos\theta_2 )^{2j_2+1}(\sin\theta_2)^{2(j_3+\cdots +j_n+(n-1)-2)+1}d\theta_2 = \frac{j_2!(j_3+\cdots +j_n+n-3)!}{2(j_2+j_3+\cdots +j_n+n-2)!}, \\
   & \vdots \\
   & \int_0^{\frac{\pi}{2}}(\cos\theta_{n-2} )^{2j_{n-2}+1}(\sin\theta_{n-2})^{2(j_{n-1} +j_n+(n-1)-(n-2))+1}d\theta_{n-2} = \frac{j_{n-2}!(j_{n-1}+j_n+1)!}{2(j_{n-2}+j_{n-1}+j_n+2)!},\\
   &  \int_0^{\frac{\pi}{2}}(\cos\theta_{n-1} )^{2j_{n-1}+1}(\sin\theta_{n-1})^{2(j_n+(n-1)-(n-1))+1}d\theta_{n-1} = \frac{j_{n-1}!j_n!}{2(j_{n-1}+j_{n}+1)!}.
   \end{align*}
   For the last integral we find
   \[
    \int_0^{+\infty} r^{2(j_1+\cdots + j_n-m)+1} (1+r^2)^{n-1}e^{-mr^2}\,dr  = \frac{(J-m)!}{2} U(J-m+1, J-m+n+1,m),
   \]
    where $J=j_1+\cdots + j_n$ and $U(a,b,z)=\frac{1}{\Gamma(a)}\int_0^{\infty}e^{-zt}t^{a-1}(1+t)^{b-a-1}\,dt$ is the Confluent Hypergeometric Function of the second kind (see \cite[page 504]{AbrSteg}).  Since
    \[
    U(a,b,z)=\frac{\Gamma(1-b)}{\Gamma(a+1-b)} {}_1F_1(a,b,z)+ \frac{\Gamma(b-1)}{\Gamma(a)}z^{1-b} {}_1F_1(a+1-b,2-b,z),
    \]
    where $_1F_1(a,b,z)$ is the {\em Confluent Hypergeometric Function}, one gets
   \begin{equation}
   ||z_1^{j_1}\cdots z_n^{j_n}||^2_{h_m} = m^{m-n-J}\frac{j_1!\cdots j_n! \Gamma(J-m+n)}{\Gamma(J+n)}{}_1F_1(1-n,1+m-n-J,m).
   \label{E0}
   \end{equation}
   Therefore
   \[
   \left\lbrace \frac{z_1^{j_1}\cdots z_n^{j_n}}{||z_1^{j_1}\cdots z_n^{j_n}||_{h_m}} \right\rbrace_{j_1+\cdots + j_m \geq m} 
   \]
   is a orthonormal basis for the Hilber Space $(\Hm,\left\langle \cdot, \cdot \right\rangle_{h_m})$. For the function epsilon (setting $J=m+k,\, k \in \N$) follows that
  \begin{equation}
  \epsilon_{mg_{S(n)}}(z)= \frac{e^{-mt}}{t^m} \sum_{k=0}^{\infty} \frac{m^{k+n}t^{k+m}(k+m+n-1)!}{(k+m)!(k+n-1)!\,{}_1F_1(1-n,1-k-n,m)},
  \label{E11}
  \end{equation}
   where $t:=|z|^2$. For $n=2$, ${}_1F_1(-1,-1-k,m)=\frac{k+m+1}{k+1}$ and \eqref{E11} simplifies to $m^2$, in agreement with the computations in Section \ref{section2}. In general, the right-hand side is, in terms of the variable $x:=mt$, equal to
 \begin{equation*}
  e^{-x}m^n \sum_{k=0}^{\infty} \frac{x^k(k+m+n-1)!}{(k+m)!(k+n-1)!\,{}_1F_1(1-n,1-k-n,m)},
  \label{E12}
  \end{equation*}
  which is independent of $x$ only if
  \begin{equation*}
  \frac{k!(k+m+n-1)!}{(k+m)!(k+n-1)!\,{}_1F_1(1-n,1-k-n,m)}
  \label{E13}
\end{equation*}
  is independent of $k$. Looking at the asymptotics as $k \rightarrow + \infty$, the last expression behaves as
  \begin{equation*}
  1 + \frac{m(n-1)(n-2)}{2k^2} + O\left(\frac{1}{k^3}\right) ,
  \label{E14}
\end{equation*}
  so it can be independent of $k$ only for $n \in \{1,2\}$.
\end{proof}


\begin{thebibliography}{1}
\bibitem{AbrSteg}  Abramowitz, M., Stegun, I. A. \textbf{Handbook of Mathematical Functions with Formulas, Graphs, and Mathematical Tables}. {Dover Publications (1972)}
\bibitem{AreLoi} Arezzo, C., Loi, A. \textbf{Quantization of K\"ahler manifolds and the asymptotic expansion of Tian-Yau-Zelditch}. {J. Geom. Phys. 47, 87-99 (2003)}
\bibitem{AreLoi3} Arezzo, C., Loi, A. \textbf{Moment maps, scalar curvature and quantization of \K manifolds}. {Comm. Math. Phys. 243, 543-559 (2004)}
\bibitem{AreLoiZhombal} Arezzo, C., Loi, A. Zuddas, F. \textbf{On homothetic balanced metrics}. {Ann.  Global Anal.  Geom. 41, n. 4, 473-491 (2012)}
\bibitem{AreLoiZ} Arezzo, C., Loi, A., Zuddas, F. \textbf{Szeg\"o kernel, regular quantizations and spherical CR-structures}. {Math. Z. 275, 1207-1216 (2013)}
\bibitem{AP09} Arezzo, C., Pacard, F. \textbf{Blowing up \K manifolds with constant scalar curvature. II}. {Ann. of Math. (2) 170, no. 2, 685-738 (2009)}
\bibitem{Berezin} Berezin, F. A. \textbf{Quantization}. Math. USSR Izvestiya 8, 1109-1163 (1974)
\bibitem{fbh}Bi, E., Feng, Z., Tu, Z. \textbf{Balanced metrics on the Fock-Bargmann-Hartogs domains}. {Ann. Glob. Anal. Geom. 49, 349-359 (2016)}
\bibitem{CahGut} Cahen, M., Gutt, S., Rawnsley, J. H. \textbf{Quantization of K\"ahler manifolds I: geometric interpretation of Berezin's quantization}. {J. Geom. Phys. 7, 45-62 (1990)}
\bibitem{Calabi} Calabi, E.  \textbf{Isometric imbeddings of complex manifolds}. {Ann. Math. 58, 1-23 (1953)}
\bibitem{FCA} Cannas Aghedu, F. \textbf{On the balanced condition for the Eguchi--Hanson metric}, {J. Geom. Phys. 137, 35-39 (2019)} 
\bibitem{DiScala} Di Scala, A., Ishi, H., Loi, A. \textbf{\K Immersions of Homogeneous \K Manifolds into Complex Space Forms}. Asian J. Math. 16 (3), 479-488 (2012)
\bibitem{Donaldson} Donaldson, S. \textbf{Scalar curvature and projective embeddings}.  {I. J. Diff. Geom. 59, 479-522 (2001)}
\bibitem{Englis} Engli\v{s}, M. \textbf{Berezin quantization and reproducing kernels on complex domains}. {Trans. Am. Math. Soc. 348, 411-479 (1996)}
\bibitem{Englis2} Engli\v{s}, M. \textbf{Weighted Bergman kernels and balanced metrics}. {RIMS Kokyuroku 1487, 40-54 (2006)}
\bibitem{fengtu} Feng Z., Tu  Z., \textbf{On canonical metrics on Cartan-Hartogs domains}, {Math. Zeit. 278, Issue 1-2, 301-320 (2014)}
\bibitem{GreLoi} Greco, A., Loi, A. \textbf{Radial balanced metrics on the unit disk}. {J. Geom. Phys. 60, 53-59 (2010)}
\bibitem{GH} Griffiths, P., Harris, J. \textbf{Principles of Algebraic Geometry}. John Wiley \& Sons, Inc. (1978)
\bibitem{Loi2} Loi, A. \textbf{The Tian-Yau-Zelditch asymptotic expansion for real analytic \K metrics}. {Int. J. Geom. Methods Mod. Phys. 1, 253-263 (2004)}
\bibitem{Loi3} Loi, A. \textbf{A Laplace integral, the T-Y-Z expansion and Berezin's transform on a \K manifold}. {Int. J. Geom. Meth. Mod. Phys. 2, 359-371 (2005)}
\bibitem{Loi4} Loi, A. \textbf{Regular quantizations of \K manifolds and constant scalar curvature metrics}. {J. Geom. Phys. 53 (3), 354-364 (2005)}
\bibitem{Loi} Loi, A. \textbf{Calabi's diastasis function for Hermitian symmetric spaces}. {Diff. Geom.Appl. 24, 311-319 (2006)}
\bibitem{LM12} Loi, A., Mossa, R. \textbf{Berezin quantization of homogeneous bounded domains}. {Geom. Dedicata 161, 119-128  (2012)}
\bibitem{LM15} Loi, A., Mossa, R. \textbf{Some remarks on Homogeneous \K manifolds}. {Geom. Dedicata Volume 179, 1-7 (2015)}
\bibitem{LMZ17} Loi, A., Mossa, R., Zuddas, F. \textbf{The log-term of  the disc  bundle over a homogeneous Hodge manifold}. Ann. Global Anal. Geom. 51  no. 1, 35-51 (2017)
\bibitem{LSZ} Loi, A., Salis, F., Zuddas, F. \textbf{Two conjectures on Ricci-flat K\"ahler metrics}. {Math. Zeit. 1, 1-15 (2018)}
\bibitem{LZ12} Loi, A., Zedda, M. \textbf{Balanced metrics on Cartan and Cartan-Hartogs domains}. {Math. Z. 270, no. 3-4, 1077-1087 (2012)}
\bibitem{taubnut} Loi, A., Zedda, M., Zuddas, F.
\textbf{Some remarks on the K\"ahler geometry of the Taub-NUT metrics}. {Ann. Global Anal. Geom. 41, no. 4, 515-533 (2012)}
\bibitem{LoiZe}Loi, A., Zedda, M. \textbf{On the coefficients of TYZ expansion of locally Hermitian symmetric spaces}. {Manuscripta math. 148, 303-315 (2015)}
\bibitem{Lu} Lu, Z. \textbf{On the lower terms of the asymptotic expansion of Tian-Yau-Zelditch}. {Am. J. Math. 122, 235-273 (2000)}
\bibitem{lutian} Lu, Z., Tian, G. \textbf{The log term of Szeg\"{o} Kernel}. Duke Math. {J. Volume 125, No 2 351-387 (2004)}
\bibitem{McDuff} McDuff, D., Salamon, D. \textbf{Introduction to Symplectic Topology}. Oxford (2017)
\bibitem{Rawnsley}Rawnsley, J. \textbf{Coherent states and \K manifolds}. {Q. J. Math. Oxford (2) 28, 403-415 (1977)}
\bibitem{Simanca} Simanca, Santiago R. \textbf{K\"ahler metrics of constant scalar curvature on bundles over $\mathbb{C}P^{n-1}$}. {Math. Ann. 291, 239-246 (1991)} 
\bibitem{Xu} Xu, H. \textbf{A closed formula for the asymptotic expansion of the Bergman kernel}. {Commun. Math. Phys. 314(3), 555-585 (2012)}
\bibitem{zedda} Zedda M. \textbf{Canonical metrics on Cartan-Hartogs domains}. {Int. J. Geom. Methods Mod. Phys. 9 , no. 1. (2012)}
\bibitem{Zelditch} Zelditch, S. \textbf{Szeg\"o Kernels and a Theorem of Tian}. {Int. Math. Res. Not. 6, 317-331 (1998)}
\end{thebibliography}
\end{document}